\documentclass[a4paper,11pt,reqno]{amsart}
\usepackage{amsxtra}
\usepackage{color}
\usepackage{amsopn}
\usepackage{amsmath,amsthm,amssymb,arydshln}
\usepackage{mathrsfs,mathtools}
\usepackage{stmaryrd}
\usepackage{graphicx}
\usepackage{bm}
\usepackage{enumerate}
\usepackage{xcolor}
\usepackage{pifont}
\usepackage{adjustbox}
\usepackage{tikz}
\usepackage{color}
\usepackage{hyperref}

\newtheorem{theorem}{Theorem}[section]

\newtheorem*{theorem*}{Theorem}

\theoremstyle{definition}
\newtheorem{definition}[theorem]{Definition}

\newtheorem{example}[theorem]{Example}

\newtheorem{question}{Question}

\newcommand{\R}{{\mathbb R}}
\renewcommand{\H}{\mathrm{H}}

\newcommand{\beq}{\begin{equation}}
\newcommand{\eeq}{\end{equation}}

\newcommand{\f}{\varphi}

\newcommand{\GL}{{\mathrm {GL}}}

\newcommand{\G}{{\mathrm G}}

\newcommand{\W}{\wedge}

\DeclareMathOperator\tr{tr}

\DeclareMathOperator\End{End}

\DeclareMathOperator\ad{ad}

\DeclareMathOperator{\Der}{Der}

\newcommand{\fra}{\mathfrak{a}}

\newcommand{\frg}{\mathfrak{g}}

\newcommand{\frh}{\mathfrak{h}}

\newcommand{\frm}{\mathfrak{m}}
\newcommand{\frn}{\mathfrak{n}}

\newcommand{\frs}{\mathfrak{s}}

\newcommand{\diag}{{\rm diag}}

\newcommand{\sst}{\scriptscriptstyle}

\numberwithin{equation}{section}

\title[Exact G$_2$-structures on unimodular Lie algebras]{Exact G$_{\mathbf2}$-structures on unimodular Lie algebras}

\author{Marisa Fern\'andez}
\address{Universidad del Pa\'{\i}s Vasco  (UPV / EHU), Facultad de Ciencia y Tecnolog\'{\i}a, 
Departamento de Matem\'aticas, Apartado 644, 48080 Bilbao, Spain}
\email{marisa.fernandez@ehu.es}
\author{Anna Fino} 
\address{Dipartimento di Matematica ``G.~Peano'' \\ Universit\`a degli Studi di Torino\\
Via Carlo Alberto 10\\
10123 Torino\\ Italy}
\email{annamaria.fino@unito.it}
\author{Alberto Raffero}
\address{Dipartimento di Matematica ``G.~Peano'' \\ Universit\`a degli Studi di Torino\\
Via Carlo Alberto 10\\
10123 Torino\\ Italy}
\email{alberto.raffero@unito.it}


\begin{document}
\begin{abstract} 
We consider seven-dimensional unimodular Lie algebras $\mathfrak{g}$ admitting exact G$_2$-structures, focusing our attention on those with vanishing third Betti number $b_3(\mathfrak{g})$.  
We discuss some examples, both in the case when $b_2(\frg)\neq0$ and in the case when the Lie algebra $\frg$ is (2,3)-trivial, i.e., when both $b_2(\mathfrak{g})$ and $b_3(\mathfrak{g})$ vanish. 
These examples are solvable, as $b_3(\mathfrak{g})=0$, but they are not strongly unimodular, 
a necessary condition for the existence of lattices on the simply connected Lie group corresponding to $\mathfrak{g}$. 
More generally, we prove that any seven-dimensional (2,3)-trivial strongly unimodular Lie algebra does not admit any exact G$_2$-structure. 
From this, it follows that there are no compact examples of the form $(\Gamma\backslash\G,\varphi)$, where $\G$ is a seven-dimensional simply connected Lie group with (2,3)-trivial Lie algebra, 
$\Gamma\subset\G$ is a co-compact discrete subgroup, and $\varphi$ is an exact $\G_2$-structure on $\Gamma\backslash\G$ induced by a left-invariant one on $\G$.
\end{abstract}

\maketitle

%%%%%%%%%%%%%%%%%%%%%%%%%%%%%%%%%%%%%%%%%%%%%%%%%%%%%%%%%%%%%%%%%%%%%%%%%%%%%%%%%%%%%%%%%
%%%%%%%%%%%%%%%%%%%%%%%%%%%%%%%%%%%%%%%%%%%%%%%%%%%%%%%%%%%%%%%%%%%%%%%%%%%%%%%%%%%%%%%%%
%																INTRODUCTION											
%%%%%%%%%%%%%%%%%%%%%%%%%%%%%%%%%%%%%%%%%%%%%%%%%%%%%%%%%%%%%%%%%%%%%%%%%%%%%%%%%%%%%%%%%
%%%%%%%%%%%%%%%%%%%%%%%%%%%%%%%%%%%%%%%%%%%%%%%%%%%%%%%%%%%%%%%%%%%%%%%%%%%%%%%%%%%%%%%%%
\section{Introduction}

Let $M$ be a seven-dimensional smooth manifold. A $\G_2$-{\em structure} on $M$ 
is a reduction of the structure group of 
its frame bundle from ${\GL}(7,\mathbb{R})$ to the  compact exceptional Lie group $\G_2$.  
In \cite{Gray}, Gray proved that a smooth $7$-manifold carries $\G_2$-structures if and only if it is orientable and spin. 

The existence of a $\G_2$-structure on $M$
is characterized by the existence of a globally defined 3-form $\f \in\Omega^3(M)$
satisfying a certain nondegeneracy condition.  
The {\em $\G_2$-form $\f$} gives rise to a Riemannian metric $g_\f$ with volume form $dV_\f$ via the identity  
\begin{equation}\label{G2metricvol}
g_\f(X,Y)\,dV_\f = \frac16\,\iota_{\sst X}\f\W\iota_{\sst Y}\f\W\f,
\end{equation}
for any pair of vector fields $X,Y$ on $M$, where $\iota_{\sst X}$ denotes the contraction by $X.$
Moreover, at each point  $ x \in M$ there exists a basis $\{ e^1, \ldots, e^7 \}$ of the cotangent space $T^*_x M$ for which
\[
\left.\varphi\right|_x =e^{127} +e^{347} +e^{567} +e^{135} - e^{146} -e^{236}-e^{245}. 
\]
Here, $e^{127}$ stands for $e^{1}\wedge e^{2}\wedge e^{7}$, and so on. Such a basis is called an {\em adapted $\G_2$-basis}.

A $\G_2$-structure is said to be {\em closed} (or {\em calibrated}) if the defining 3-form $\f$ satisfies the equation $d\f=0$, 
while it is called {\em co-closed} (or {\em co-calibrated}) if $d*_\f\f=0$, $*_\f$ being the Hodge operator defined by $g_\f$ and $dV_\f$.  
When both of these conditions hold, the intrinsic torsion of the $\G_2$-structure vanishes identically, the Riemannian metric $g_\f$ is Ricci-flat, 
and $\mathrm{Hol}(g_\f)\subseteq\G_2$ (cf.~\cite{Bryant-1,FerGray}). In this case, the $\G_2$-structure is said to be {\em torsion-free}.

The existence of a torsion-free $\G_2$-structure on a compact 7-manifold $M$ imposes various constraints on the topology. For instance, the
third Betti number must satisfy $b_3(M) \geq 1$ \cite{Bonan}, and  the first Betti number $b_1(M)\in \{0, 1, 3, 7\}$  \cite{Joyce1}. 
Moreover, if  $\mathrm{Hol}(g_\f) = \G_2$, then the fundamental group $\pi_1(M)$ is finite \cite{Joyce1}, and so $b_1(M) = 0$. 
However, a non-compact manifold with a torsion-free $\G_2$-structure may have vanishing third Betti number.  
An example is given by the metric cone of the nearly K\"ahler flag manifold ${\mathbb F}_{1,2}$ \cite{Bryant-1}. 

In the literature, all known examples of compact 7-manifolds $M$ admitting a closed $\G_2$-structure, but not torsion-free $\G_2$-structures,
have $b_1 (M) > 0$ and  $b_3 (M) > 0$ (see \cite{CF, Fer, Fer1, FFKM}). 
A longstanding open question concerns the existence of closed $\G_2$-structures on compact 7-manifolds with $b_3 (M)=0$, such as the 7-sphere. 
Notice that, in this case, any closed $\G_2$-structure would be defined by an exact 3-form. 

Motivated by this problem, one may investigate the existence of such type of examples when $M=\Gamma\backslash\G$, 
where $\G$ is a seven-dimensional simply connected Lie group, $\Gamma\subset\G$ is a co-compact discrete subgroup (lattice), and $M$ is endowed with an {\em exact} $\G_2$-structure induced by a 
left-invariant one on $\G$.  

Recall that the de Rham cohomology groups of $\Gamma\backslash\G$ are isomorphic to those of the Chevalley-Eilenberg complex of the Lie algebra $\frg$ of $\G$ when, for instance,  
the latter is nilpotent or completely solvable (cf.~\cite{Nom,Hat}). 
Thus, a possible strategy to study the problem consists in looking for suitable Lie algebras $\frg$ with third Betti number $b_3(\frg)=0$ and admitting closed $\G_2$-structures, 
and see whether the corresponding simply connected Lie group admits a lattice.  
Since seven-dimensional nilpotent Lie algebras cannot admit exact $\G_2$-structures (cf.~\cite{CF}), and since any Lie algebra satisfying $b_3(\frg)=0$ is solvable and not nilpotent (see e.g.~\cite{MaSw1}), 
one may firstly focus on the latter case. 
By a result of Garland \cite{Gar}, every simply connected solvable Lie group containing a lattice must be strongly unimodular (see Definition \ref{def:stronglyunim} for details). 
This provides a further restriction on $\frg$. 
 
In Section \ref{sect:(2,3)-trivial}, we focus on Lie algebras $\frg$ satisfying the conditions $b_2(\frg)=0=b_3(\frg)$. 
They are known as (2,3)-{\em trivial} in the literature, and they are always solvable with codimension one derived algebra $[\frg,\frg]$ \cite{MaSw1}. 

In Example \ref{ex23trivial}, we show the existence of a unimodular (2,3)-trivial Lie algebra $\frs$ admitting closed $\G_2$-structures. 
Since $\frs$ is unimodular, by Hodge duality it also satisfies $b_4(\frs)=0=b_5(\frs)$. 
However, we prove that it is not strongly unimodular, thus the corresponding simply connected solvable Lie group does not admit lattices. 
Notice that this example reveals an interesting difference between closed $\G_2$-structures and symplectic structures (see \cite{FiRa} for further results in this direction). 
Indeed, both structures are defined by closed differential forms satisfying the same non-degeneracy condition, namely their orbit under the action of the general linear group is open, 
but a unimodular Lie algebra cannot admit exact symplectic forms (cf.~\cite{DiBa}). 

More generally, we prove that it is not possible to obtain compact examples of the form $M=\Gamma\backslash\G$, 
when the Lie algebra of the simply connected Lie group $\G$ is (2,3)-trivial and strongly unimodular.   
In detail, using a characterisation of (2,3)-trivial Lie algebras by Madsen and Swann \cite{MaSw1}, 
we first obtain the classification of all seven-dimensional (2,3)-trivial strongly unimodular Lie algebras up to isomorphism (Theorem \ref{Thm23trivial}). 
Then, we show that none of them admits exact $\G_2$-structures (Theorem \ref{Thm23trivialNoexact}). 
 
In Section \ref{sect:lie-group}, we discuss an example of a seven-dimensional unimodular Lie algebra $\frh$ satisfying $b_3(\frh)=0$, $b_2(\frh)\neq0$ and admitting closed $\G_2$-structures.  
Also in this case, it turns out that the solvable Lie algebra $\frh$ is not strongly unimodular. 
This leads to the following open problem: {\em is there any seven-dimensional Lie algebra $\frg$ satisfying $b_3(\frg)=0$ which is strongly unimodular and admits closed $\G_2$-structures?}

%%%%%%%%%%%%%%%%%%%%%%%%%%%%%%%%%%%%%%%%%%%%%%%%%%%%%%%%%%%%%%%%%%%%%%%%%%%%%%%%%%%%%%%%%
%%%%%%%%%%%%%%%%%%%%%%%%%%%%%%%%%%%%%%%%%%%%%%%%%%%%%%%%%%%%%%%%%%%%%%%%%%%%%%%%%%%%%%%%%
%														UNIMODULAR EXAMPLES										
%%%%%%%%%%%%%%%%%%%%%%%%%%%%%%%%%%%%%%%%%%%%%%%%%%%%%%%%%%%%%%%%%%%%%%%%%%%%%%%%%%%%%%%%%
%%%%%%%%%%%%%%%%%%%%%%%%%%%%%%%%%%%%%%%%%%%%%%%%%%%%%%%%%%%%%%%%%%%%%%%%%%%%%%%%%%%%%%%%%
\section{(2,3)-trivial strongly unimodular Lie algebras} \label{sect:(2,3)-trivial}
In this section, we focus on seven-dimensional (2,3)-trivial Lie algebras.  
First of all, we recall that being (2,3)-trivial imposes strong constraints on the structure of a Lie algebra. 
Indeed, as shown in \cite{MaSw1}, any finite-dimensional Lie algebra $\frg$ different from $\R,\R^2$ and satisfying $b_3(\frg)=0$ is solvable and not nilpotent. 
Moreover, if also $b_2(\frg)=0$, then the derived algebra $\frn\coloneqq[\frg,\frg]$ is a codimension one ideal, and $\frg$ is the semidirect product $\frg = \frn\rtimes\R$. 
We emphasise that $\frn$ is nilpotent, as $\frg$ is solvable.

In the next example, we describe a seven-dimensional $(2,3)$-trivial unimodular Lie algebra admitting closed $\G_2$-structures. 
This shows in particular that exact $\G_2$-structures behave differently from exact symplectic structures, as the latter do not exist on unimodular Lie algebras (cf.~\cite{DiBa}). 

\begin{example}\label{ex23trivial}
Let $\frs=\langle e_1,e_2,e_3,e_4,e_5,e_6,e_7\rangle$ be the seven-dimensional solvable Lie algebra with the following nonzero Lie brackets
\[
\renewcommand\arraystretch{1.2}
\begin{array}{llll}
[e_1, e_3] = e_4 - 3 e_6,			&	[e_1, e_4] =  e_5,			&	[e_1,e_5] = -e_6,				& 	[e_2, e_3] =  e_5,							\\[2pt]
[e_2, e_4] =  -e_6, 				&	[e_1, e_7] = 2 e_1,			&	[e_2, e_7] = 4 e_2,				&	[e_3, e_7] = -\frac92 e_3+6 e_5,				\\[2pt]				
[e_4, e_7] =  -\frac52 e_4 + 6 e_6,	&	[e_5, e_7] = -\frac12 e_5,		&	[e_6, e_7] =  \frac32 e_6.			&											\\ [2pt]			
\end{array}
\]
Notice that $\frs$ is the semidirect product $\frs=\frn\rtimes\R,$ where $\R=\langle e_7\rangle$ and $\frn=\langle e_1,\ldots,e_6\rangle$ is isomorphic to the unique six-dimensional 4-step 
nilpotent Lie algebra with $b_1(\frn)=3$ and $b_2(\frn)=4$ (cf.~\cite[Table A.1]{Sal}).  
Moreover, $\frs$ is completely solvable, i.e., for each $X\in\frs$ the map $\ad_{\sst X}\in\End(\frs)$ has only real eigenvalues. 

Let $(e^1,\ldots,e^7)$ be the dual basis of $(e_1,\ldots, e_7)$. 
Then, the structure equations of $\frs$ can also be written by means of the Chevalley-Eilenberg differentials of the covectors $e^i$. In detail: 
\[
\begin{split}
&de^1 = -2 e^{17},\quad de^2 = -4 e^{27},\quad de^3 = \frac92 e^{37}, \quad de^4 = \frac52 e^{47} - e^{13},	\\
&de^5 = \frac12 e^{57} - 6 e^{37} - e^{14} - e^{23},\quad de^6	= -\frac32 e^{67} - 6 e^{47} + 3 e^{13} + e^{15} + e^{24}, \quad de^7=0.
\end{split}
\]
Using these equations, it is possible to compute the cohomology groups of the Chevalley-Eilenberg complex $(\Lambda^\bullet(\frs^*),d)$ of $\frs$, obtaining 
\[
H^1(\frs) = \langle [e^7]\rangle, \quad H^6(\frs) = \langle [e^{123456}] \rangle, \quad H^7(\frs) = \langle [e^{1234567}] \rangle, \quad H^i(\frs) = 0,~i=2,3,4,5. 
\]
We immediately see that $b_2(\frs)=0=b_3(\frs)$, hence $\frs$ is (2,3)-trivial. 
Moreover, since $H^7(\frs)\cong \R$, the Lie algebra $\frs$ is unimodular. Recall that this is equivalent to having $\tr(\ad_{\sst X})=0$ for each $X\in\frs$.  

The Lie algebra $\frs$ admits closed (hence exact) $\G_2$-structures. An example is given by the following 3-form
\[
\f = e^{127} +e^{347} +e^{567} +e^{135} - e^{146} -e^{236}-e^{245} = d\left(\frac16 e^{12} + \frac{23}{7} e^{34} + 2 e^{36} - 2 e^{45} + e^{56} \right). 
\]

Consequently, the simply connected solvable Lie group $\mathrm{S}$ with Lie algebra $\frs$ is endowed with a {\em left-invariant} exact $\G_2$-structure obtained from $\f$ via left multiplication. 
\end{example}

It is well-known that a Lie group admits a lattice only if it is unimodular (cf.~\cite{Mil}). 
When the Lie group is solvable, there is a stronger necessary condition for the existence of lattices, namely the group must be {\em strongly unimodular} (cf.~\cite[Prop.~3.3]{Gar}). 
We recall the definition here. 

\begin{definition}\label{def:stronglyunim}
Let $\mathrm{G}$ be a simply connected solvable Lie group with Lie algebra $\frg$. 
Denote by $\frn^{\sst0}\coloneqq\frn$ the nilradical of $\frg$ and, for each positive integer $i\geq 1$, let $\frn^{\sst i} \coloneqq [\frn,\frn^{\sst i-1}]$ 
denote the $i^{th}$ term in the descending central series of $\frn$. 
The Lie algebra $\frg$ is {\em strongly unimodular} if for all $X\in\frg$ the restriction of $\ad_{\sst X}$ to each space $\frn^{\sst i}/\frn^{\sst i+1}$ is traceless. 
In this case, the Lie group $\mathrm{G}$ is said to be {\em strongly unimodular}.
\end{definition}

Notice that any solvable strongly unimodular Lie algebra $\frg$ is unimodular. Indeed, if the nilradical $\frn$ of $\frg$ is $s$-step nilpotent, then we have the vector space decomposition
\begin{equation}\label{susplit}
\frg = \frg / \frn \oplus  \frn/\frn^{\sst1} \oplus \frn^{\sst1}/\frn^{\sst2} \oplus \cdots \oplus \frn^{\sst s-2}/\frn^{\sst s-1} \oplus  \frn^{\sst s-1}, 
\end{equation}
where $\frg / \frn$ is abelian since $[\frg,\frg]\subseteq\frn$. 
Now, for every $X\in\frg$, the endomorphism $\ad_{\sst X}\in\End(\frg)$ preserves every summand of the splitting \eqref{susplit}. 
From this and the Definition \ref{def:stronglyunim}, we immediately get that $\ad_{\sst X}$ is traceless, i.e., $\frg$ is unimodular. 
On the other hand, a unimodular solvable Lie algebra may not be strongly unimodular. This is the case, for instance, of the Lie algebra $\frs$ considered in Example \ref{ex23trivial}. 
Indeed, its nilradical $\frn$ has the following descending central series 
\[
\frn^{\sst0}=\frn,\quad \frn^{\sst1} = \langle e_4-3 e_6,e_5,e_6\rangle,\quad \frn^{\sst2} = \langle e_5, e_6\rangle, \quad \frn^{\sst3} = \langle e_6\rangle,\quad \frn^{\sst4} = \{0\}, 
\]
and the restriction of $\ad_{e_7}$ to $\frn^{\sst3}$ is not traceless. This shows in particular that the simply connected solvable Lie group $\mathrm{S}$ does not contain any lattice. 

It is then natural to ask the following. 
\begin{question}
Is there any solvable strongly unimodular Lie algebra admitting exact $\G_2$-structures?
\end{question}

As we are interested in examples having as many zero Betti numbers as possible, we investigate this problem under the assumption that the Lie algebra is $(2,3)$-trivial. 
Our strategy is the following: we first determine all possible (2,3)-trivial, strongly unimodular Lie algebras of dimension seven up to isomorphism, 
and then we investigate whether any of them admits exact $\G_2$-structures. 
In the next theorem, we deal with the first problem. 
\begin{theorem}\label{Thm23trivial}
A (2,3)-trivial, strongly unimodular seven-dimensional Lie algebra is solvable and its six-dimensional nilradical is either abelian or isomorphic to one of the following 
nilpotent Lie algebras
\[
\left(0,0,0,0,e^{12},e^{34}\right),\quad \left(0,0,0,0,e^{13}-e^{24},e^{14}+e^{23}\right). 
\]
\end{theorem}

\begin{proof} 
By \cite[Thm.~4.3]{MaSw1}, a Lie algebra $\frg$ with derived algebra $\frn\coloneqq[\frg,\frg]$ is (2,3)-trivial if and only if $\frg$ is solvable, $\frn$ is nilpotent of codimension one in $\frg$ and  
$H^1(\frn)^\frg = \{0\} = H^2(\frn)^\frg = H^3(\frn)^\frg$. 
Here, $H^k(\frn)$ denotes the $k^{th}$ cohomology group of the Chevalley-Eilenberg complex $\left(\Lambda^\bullet(\frn^*),\hat{d}\right)$ of $\frn$, and 
$H^k(\frn)^\frg$ denotes the fixed point set of the $\frg$-action on $H^k(\frn)$ for which $\frn\subset\frg$ acts trivially and every $X\in \frg/\frn\cong\R$ acts as follows: 
\[
X\cdot [\alpha] = [\iota_{\sst X} d\alpha], \quad \forall~[\alpha]\in H^k(\frn),
\]
where $d$ is the Chevalley-Eilenberg differential on $\frg$. 

This characterisation leads us to considering seven-dimensional solvable Lie algebras of the form $\frg = \frn \rtimes_D \R$, where $\frn$ is a six-dimensional nilpotent Lie algebra 
and $\R=\langle X\rangle$ with $\ad_{\sst X} = D\in\Der(\frn)$. To determine all possible Lie algebras of this type, up to isomorphism, we proceed as follows. 
It is known that there exist 34 six-dimensional nilpotent Lie algebras up to isomorphism.  
For each of them, we consider the basis $\{e_1,\ldots, e_6\}$ of $\frn$ for which its structure equations are those given in \cite[Table A.1]{Sal}. 
We can then compute the generic derivation $D\in\Der(\frn)$ and consider the $6\times6$ matrix associated to it with respect to the basis $\{e_1,\ldots, e_6\}$. 
This matrix may have up to 36 unknown entries, which we denote by $a_i$. Finally, we let  $\R = \langle e_7 \rangle$ with $\ad_{e_7} = D$. 

Now, for any six-dimensional nilpotent Lie algebra $\frn$, we have to check whether the Lie algebra $\frg = \frn \rtimes_D \R = \langle e_1,\ldots,e_7\rangle$ 
can both be strongly unimodular and satisfy $H^1(\frn)^\frg = \{0\} = H^2(\frn)^\frg = H^3(\frn)^\frg$. 
Clearly, $\frg$ is strongly unimodular if and only if the restriction of $\ad_{e_7} = D$ to each space $\frn^{\sst i}/\frn^{\sst i+1}$ is traceless. 
This gives a set of linear equations in the real variables $a_i$, say $\mathcal{S}$. 
As for the second condition, we first determine a basis of the cohomology groups $H^k(\frn)$, $k=1,2,3$, 
and then we compute the square matrix $A_k$ associated to the linear endomorphism $H^k(\frn)\ni[\alpha]\mapsto e_7\cdot[\alpha]$ with respect to that basis. 
The entries of these matrices are polynomials in the variables $a_i$, and the condition $H^k(\frn)^\frg = \{0\}$ is equivalent to $A_k$ being non-singular.  

Summing up, for each six-dimensional nilpotent Lie algebra $\frn$, we have to solve the set of linear equations $\mathcal{S}$ under the constraints $\det(A_k)\neq0,~k=1,2,3.$ 
For 31 out of 34 non-isomorphic six-dimensional nilpotent Lie algebras this is not possible. 

As an example, let us consider the first nilpotent Lie algebra appearing in \cite[Table A.1]{Sal}. 
Its structure equations with respect to the dual basis $\{e^1,\ldots,e^6\}$ of $\{e_1,\ldots,e_6\}$ are the following
\[
\frn = \left(0,0,e^{12},e^{13},e^{14}+e^{23},e^{34}-e^{25}\right). 
\]
The descending central series of $\frn$ is 
\[ 
\frn^{\sst1} = \langle e_3,e_4,e_5,e_6 \rangle,\quad \frn^{\sst2} = \langle e_4,e_5,e_6 \rangle,\quad \frn^{\sst3} = \langle e_5,e_6 \rangle, \quad \frn^{\sst4} = \langle e_6 \rangle, \quad \frn^{\sst5} = \{0\},  
\]
and the matrix associated to a generic derivation $D\in\Der(\frn)$ with respect to the basis $\{e_1,\ldots,e_6\}$ is the following 
\[
D = 
\begin{pmatrix}
a_1	&	0	&	0	&	0	&	0	&	0	\\
0	&	2a_1	&	0	&	0	&	0	&	0	\\
a_2	&	a_3	&	3a_1	&	0	&	0	&	0	\\
a_4	&	0	&	a_3	&	4a_1	&	0	&	0	\\
a_5	&	a_6	&     -a_2	&	a_3	&	5a_1	&	0	\\
a_7	&	a_8	&     a_5	&	-a_4	&	a_2	&	7a_1
\end{pmatrix}. 
\]
Now, we see that $\frg = \frn\rtimes_D\R$ is strongly unimodular if and only if $a_1=0$. 
On the other hand, we have $H^1(\frn) = \langle [e^1], [e^2] \rangle$, and we see that 
\[
e_7\cdot[e^1] = [\iota_{e_7}de^1] = -a_1 [e^1], \quad e_7\cdot[e^2] = [\iota_{e_7}de^2] = -2a_1 [e^2].
\] 
Thus, $\det(A_1)=2(a_1)^2$, whence it follows that $\frg$ cannot be both (2,3)-trivial and strongly unimodular. 
In the remaining 30 cases, we proceed similarly. 
We obtain that any solution of $\mathcal{S}$ implies $\det(A_1)=0$ when $\frn$ is one of the following Lie algebras
\[
\renewcommand\arraystretch{1.4}
\begin{array}{cc}
\left(0,0,e^{12},e^{13},e^{14},e^{34}-e^{25}\right),		&	\left(0,0,e^{12},e^{13},e^{14},e^{15}\right),				\\
\left(0,0,e^{12},e^{13},e^{14}+e^{23},e^{15}+e^{24}\right), 	&	\left(0,0,e^{12},e^{13},e^{14},e^{15}+e^{23}\right),		\\	
\left(0,0,e^{12},e^{13},e^{23},e^{14}\right),				&	\left(0,0,0,e^{12},e^{14}-e^{23},e^{15}+e^{34}\right), 		\\ 
\left(0,0,0,e^{12},e^{14},e^{15}+e^{23}\right), 	 		&	\left(0,0,0,e^{12},e^{14},e^{15}+e^{23}+e^{24}\right),		\\
\left(0,0,0,e^{12},e^{14},e^{15}+e^{24}\right), 			&	\left(0,0,0,e^{12},e^{14},e^{15}\right), 				\\
\left(0,0,0,e^{12},e^{13},e^{14}+e^{35}\right), 			&	\left(0,0,0,e^{12},e^{23},e^{14}+e^{35}\right), 			\\
\left(0,0,0,e^{12},e^{23},e^{14}-e^{35}\right), 			&	\left(0,0,0,e^{12},e^{14},e^{24}\right), 				\\	
\left(0,0,0,e^{12},e^{13}-e^{24},e^{14}+e^{23}\right), 		&	\left(0,0,0,e^{12},e^{14},e^{13}-e^{24}\right), 			\\
\left(0,0,0,e^{12},e^{13}+e^{14},e^{24}\right), 			&	\left(0,0,0,e^{12},e^{13},e^{14}+e^{23}\right), 			\\
\left(0,0,0,e^{12},e^{13},e^{24}\right), 				&	\left(0,0,0,e^{12},e^{13},e^{14}\right), 				\\
\left(0,0,0,0,e^{12},e^{15}+e^{34}\right),				&	\left(0,0,0,0,e^{12},e^{15}\right),						\\
\left(0,0,0,0,e^{12},e^{14}+e^{25}\right),				&	\left(0,0,0,0,0,e^{12}+e^{34}\right).	
\end{array}
\]
Any solution of $\mathcal{S}$ implies $\det(A_2)=0$ in the following cases
\[
\begin{array}{ll}
\left(0,0,0,e^{12},e^{13},e^{23}\right),	&	\left(0,0,0,0,e^{12},e^{14}+e^{23}\right).
\end{array}
\]
Finally, any solution of $\mathcal{S}$ implies $\det(A_3)=0$ when $\frn$ is one of the following
\[
\renewcommand\arraystretch{1.4}
\begin{array}{cc}
\left(0,0,e^{12},e^{13},e^{23},e^{14}-e^{25}\right),	&	\left(0,0,e^{12},e^{13},e^{23},e^{14}+e^{25}\right),\\	
\left(0,0,0,0,e^{12},e^{13}\right),					&	\left(0,0,0,0,0,e^{12}\right),	
\end{array}
\]

We now claim that there exist (2,3)-trivial and strongly unimodular Lie algebras $\frg=\frn\rtimes_D\R$ when $\frn$ is one of the Lie algebras of \cite[Table A.1]{Sal} not appearing above, namely 
\[
\fra = \left(0,0,0,0,0,0\right),\quad \frn_{\sst1}=\left(0,0,0,0,e^{12},e^{34}\right),\quad \frn_{\sst2}=\left(0,0,0,0,e^{13}-e^{24},e^{14}+e^{23}\right). 
\]
Let us examine each case separately. 

For the abelian Lie algebra $\fra$, we have $\Der(\fra)\cong \End(\fra)$ and $H^k(\fra)\cong\Lambda^k(\fra^*)$.  
To obtain a (2,3)-trivial Lie algebra $\frg = \fra \rtimes_D\R$, it is sufficient to consider a diagonal derivation 
\[
D = \diag\left(a_1,a_2,a_3,a_4,a_5,a_6\right),
\]
with $a_i\neq0$, and $a_i\neq -a_j, -a_j-a_k$, whenever $i, j, k$ are distinct (see also \cite[Ex.~5.1]{MaSw2} for a similar case). 
Moreover, the Lie algebra $\frg$ is strongly unimodular if and only if $\tr(D) = \sum_{i=1}^6 a_i = 0$. This last equation clearly admits solutions under the above contraints. 

The Lie algebra $\frn_{\sst1}$ has the following derived series 
\[
(\frn_{\sst1})^{\sst0}=\frn_{\sst1},\quad (\frn_{\sst1})^{\sst1} = \langle e_5, e_6\rangle, \quad (\frn_{\sst1})^{\sst i} = \{0\},~i\geq2, 
\]
and the generic derivation $D_1\in\Der(\frn_{\sst1})$ has the following expression
\begin{equation}\label{Dern1}
D_1 = 
\begin{pmatrix}
a_1		&	a_2	&	0	&	0	&	0	&	0	\\
a_3		&	a_4	&	0	&	0	&	0	&	0	\\
0		&	0	&	a_5	&	a_6	&	0	&	0	\\
0		&	0	&	a_7	&	a_8	&	0	&	0	\\
a_9		&a_{10}	&a_{11}	&a_{12}	&a_1+a_4	&	0	\\
a_{13}	&a_{14}	&a_{15}	&a_{16}	&	0	&a_5+a_8
\end{pmatrix}. 
\end{equation}
Thus, the seven-dimensional Lie algebra $\frg = \frn_{\sst1} \rtimes_{D_1} \R$ is strongly unimodular if and only if 
\[
a_1+a_4+a_5+a_8=0. 
\] 
We can now compute the determinant of the matrices $A_1, A_2, A_3$  (see the Appendix for their expression), 
and notice that there exist solutions of the equation above for which $\det(A_k)\neq0,$ for $k=1,2,3$. 
An example is given by  
\[
a_1 = 1,\quad a_4 = 3,\quad a_5 = 2,\quad a_8 = -6,\quad a_i = 0~\mbox{otherwise}. 
\]
  
Finally, let us consider the Lie algebra $\frn_{\sst2}$. Its derived series is  
\[
(\frn_{\sst2})^{\sst0}=\frn_{\sst2},\quad (\frn_{\sst2})^{\sst1} = \langle e_5, e_6\rangle, \quad (\frn_{\sst2})^{\sst i} = \{0\},~i\geq2, 
\]
and the generic derivation $D_2\in\Der(\frn_{\sst2})$ is
\begin{equation}\label{Dern2}
D_2 = 
\begin{pmatrix}
a_1		&	a_2	&	a_3	&	a_4	&	0	&	0	\\
-a_2		&	a_1	&	-a_4	&	a_3	&	0	&	0	\\
a_5		&	a_6	&	a_7	&	a_8	&	0	&	0	\\
-a_6		&	a_5	&	-a_8	&	a_7	&	0	&	0	\\
a_9		&a_{10}	&a_{11}	&a_{12}	&a_1+a_7	&a_2+a_8	\\
a_{13}	&a_{14}	&a_{15}	&a_{16}	&-a_2-a_8	&a_1+a_7
\end{pmatrix}. 
\end{equation}
From this, we see that $\frg = \frn_{\sst2}\rtimes_{D_2}\R$ is strongly unimodular if and only if 
\[
a_1+a_7=0. 
\]
There exist solutions of this equation such that $\det(A_k)\neq0,$ for $k=1,2,3$. For instance 
\[
a_1=1,\quad a_4 = 2,\quad a_5 = 2,\quad a_7=-1, \quad a_8=1,\quad a_i = 0~\mbox{otherwise}. 
\]
Also in this case, the explicit expression of the matrices $A_1,A_2,A_3$ can be found in the Appendix. 
\end{proof}

We now show that there are no exact $\G_2$-structures on any seven-dimensional solvable Lie algebra of the form $\frg = \frn \rtimes_D \R$, when $\frn = \fra, \frn_{\sst1}$, and when $\frn=\frn_{\sst2}$ and 
$D\in\Der(\frn_{\sst2})$ is such that $\frg$ is strongly unimodular.     
Combining this result with Theorem \ref{Thm23trivial}, we get that there are no (2,3)-trivial, strongly unimodular seven-dimensional Lie algebras admitting exact $\G_2$-structures. 

\begin{theorem}\label{Thm23trivialNoexact}
Any seven-dimensional solvable Lie algebra $\frg=\frn\rtimes_D\R$, with $\frn=\fra, \frn_{\sst1}$, does not admit exact $\G_2$-structures. 
Moreover, the same conclusion holds true if $\frn=\frn_{\sst2}$ and the derivation $D\in\Der(\frn_{\sst2})$ is such that $\frg$ is strongly unimodular. 
\end{theorem} 

\begin{proof}
Recall that a 3-form $\phi$ on a seven-dimensional Lie algebra $\frg$ defines a $\G_2$-structure if and only if 
\begin{equation}\label{nondegenerate}
\iota_{\sst v}\phi \W \iota_{\sst v}\phi \W \phi \neq 0 \in \Lambda^7(\frg^*),\quad \forall~v\in\frg\smallsetminus\{0\}.
\end{equation}

We have to consider the seven-dimensional solvable Lie algebra $\frg=\frn\rtimes_D\R$ in the case when the nilradical $\frn=\langle e_1,\ldots,e_6\rangle$ is one of $\fra, \frn_{\sst1}, \frn_{\sst2},$ 
with the structure equations given in the proof of Theorem \ref{Thm23trivial}, and $\R$ is generated by a vector $e_7$ such that $\ad_{e_7}$ is a generic derivation $D\in\Der(\frn)$. 

Under these assumptions, we need to show that there are no exact 3-forms satisfying the condition \eqref{nondegenerate} when $\frn=\fra, \frn_{\sst1},$ 
and when $\frn=\frn_{\sst2}$ with $D\in\Der(\frn_{\sst2})$ such that $\frg$ is strongly unimodular. 
In what follows, we denote the Chevalley-Eilenberg differential on $\frg$ and on $\frn$ by $d$ and $\hat{d}$, respectively. 
Recall that for any 2-form $\gamma\in\Lambda^k(\frn^*)$ the following identity holds
\[
d\gamma = \hat{d}\gamma +(-1)^{k+1}\,D^*\gamma \W e^7,
\]
where $D^*\gamma(v_1,\ldots,v_k) = \gamma(Dv_1,\ldots,v_k)+\cdots+\gamma(v_1,\ldots,Dv_k)$, for all $v_1,\ldots,v_k\in\frn$. 

Let $\frn=\fra$ be the abelian Lie algebra, and consider the generic 2-form $\alpha = \beta + \eta\W e^7$ on $\frg$, where $\beta\in\Lambda^2(\fra^*)$ and $\eta \in \Lambda^1(\fra^*)$. 
Then, the generic exact 3-form on $\frg$ is given by 
\[
\phi = d\alpha = d\beta + d\eta\W e^7 = \hat{d}\beta - D^*\beta\W e^7 +\hat{d}\eta\W e^7 = - D^*\beta\W e^7, 
\]
as $\hat{d}\gamma=0$ for any $\gamma\in\Lambda^k(\fra^*)$. 
It is clear that this 3-form cannot define a $\G_2$-structure, since there exists some nonzero $v\in\fra$ such that $\iota_{\sst v}\f \W \iota_{\sst v}\f \W \f =0$. 

Let us now consider the Lie algebra $\frg = \frn_{\sst1}\rtimes_{D_1}\R$, where $D_1\in\Der(\frn_{\sst1})$ is given by \eqref{Dern1}. 
Let $\alpha\in\Lambda^2(\frg^*)$, we write $\alpha = \sum_{1\leq i < j \leq 7} c_{ij} e^i \W e^j$, with $c_{ij}\in\R$, and see that a generic exact 3-form on $\frg$ has the following expression 
\[
\begin{split}
d\alpha 	&= 	(-a_1c_{12}-a_{4}c_{12}+a_{9}c_{25}-a_{10}c_{15}+a_{13}c_{26}-a_{14}c_{16}+c_{57}) e^{127} - c_{16}e^{134}\\
		&	-(a_{1}c_{16}+a_{3}c_{26}+a_{5}c_{16}+a_{8}c_{16}+a_{9}c_{56})e^{167} -c_{26} e^{234} -c_{35} e^{123} +c_{56}e^{126}\\ 
		&	-(a_{2}c_{16}+a_{4}c_{26}+a_{5}c_{26}+a_{8}c_{26}+a_{10}c_{56}) e^{267} -c_{45} e^{124} -c_{56} e^{345}\\
		&	-(a_{1}c_{13}+a_{3}c_{23}+a_{5}c_{13}+a_{7}c_{14}-a_{9}c_{35}+a_{11}c_{15}-a_{13}c_{36}+a_{15}c_{16}) e^{137}\\
		&	-(a_{2}c_{13}+a_{4}c_{23}+a_{5}c_{23}+a_{7}c_{24}-a_{10}c_{35}+a_{11}c_{25}-a_{14}c_{36}+a_{15}c_{26}) e^{237}\\
		&	-(2a_{5}c_{36}+a_{7}c_{46}+a_{8}c_{36}+a_{11}c_{56}) e^{367} -(a_{1}c_{25}+a_{2}c_{15}+2a_{4}c_{25}-a_{14}c_{56}) e^{257} \\
		&	-(a_{1}c_{14}+a_{3}c_{24}+a_{6}c_{13}+a_{8}c_{14}-a_{9}c_{45}+a_{12}c_{15}-a_{13}c_{46}+a_{16}c_{16}) e^{147}\\
		&	-(a_{2}c_{14}+a_{4}c_{24}+a_{6}c_{23}+a_{8}c_{24}-a_{10}c_{45}+a_{12}c_{25}-a_{14}c_{46}+a_{16}c_{26}) e^{247}\\
		&	-(a_{5}c_{34}+a_{8}c_{34}-a_{11}c_{45}+a_{12}c_{35}-a_{15}c_{46}+a_{16}c_{36}-c_{67}) e^{347}\\
		&	-(a_{5}c_{46}+a_{6}c_{36}+2a_{8}c_{46}+a_{12}c_{56}) e^{467} -(2a_{1}c_{15}+a_{3}c_{25}+a_{4}c_{15}-a_{13}c_{56}) e^{157}\\
		&	-(a_{1}c_{35}+a_{4}c_{35}+a_{5}c_{35}+a_{7}c_{45}-a_{15}c_{56}) e^{357}  -c_{56}(a_{4}+a_{1}+a_{8}+a_{5}) e^{567}\\
		&	-(a_{1}c_{45}+a_{4}c_{45}+a_{6}c_{35}+a_{8}c_{45}-a_{16}c_{56}) e^{457}. 
\end{split}
\]
This exact 3-form does not define a $\G_2$-structure, as we always have $\iota_{{e_6}}d\alpha \W \iota_{{e_6}}d\alpha \W d\alpha=0$, for all $a_i$ and $c_{ij}$. 

We are left with the case $\frg = \frn_{\sst2}\rtimes_{D_2}\R$. 
Recall from the proof of Theorem \ref{Thm23trivial} that $\frg$ is strongly unimodular if and only if the entries of the derivation $D_2\in\Der(\frn_{\sst2})$ given in \eqref{Dern2} satisfy the condition $a_1+a_7=0.$ 
Now, similarly as in the previous case, we can compute the expression of the generic exact 3-form $d\alpha\in\Lambda^3(\frg^*)$ and see that 
\[
\iota_{{e_6}}d\alpha \W \iota_{{e_6}}d\alpha \W d\alpha = -12\,(c_{56})^3 \left(a_1+a_7\right) e^{1234567}. 
\]
Thus, if $\frg$ is strongly unimodular, then $d\alpha$ cannot define a $\G_2$-structure. 
\end{proof}

%%%%%%%%%%%%%%%%%%%%%%%%%%%%%%%%%%%%%%%%%%%%%%%%%%%%%%%%%%%%%%%%%%%%%%%%%%%%%%%%%%%%%%%%%
%%%%%%%%%%%%%%%%%%%%%%%%%%%%%%%%%%%%%%%%%%%%%%%%%%%%%%%%%%%%%%%%%%%%%%%%%%%%%%%%%%%%%%%%%
%														LIE ALGEBRA AND GROUP											
%%%%%%%%%%%%%%%%%%%%%%%%%%%%%%%%%%%%%%%%%%%%%%%%%%%%%%%%%%%%%%%%%%%%%%%%%%%%%%%%%%%%%%%%%
%%%%%%%%%%%%%%%%%%%%%%%%%%%%%%%%%%%%%%%%%%%%%%%%%%%%%%%%%%%%%%%%%%%%%%%%%%%%%%%%%%%%%%%%%
\section{A unimodular example with $b_2\neq0$ and $b_3=0$}\label{sect:lie-group}
Let $\frh=\langle e_1,e_2,e_3,e_4,e_5,e_6,e_7\rangle$ be the seven-dimensional solvable Lie algebra with the following nonzero Lie brackets
\begin{equation}\label{StrEqns}
\renewcommand\arraystretch{1.2}
\begin{array}{lllll}
[e_1, e_3] = 4\, e_3,	&	[e_1, e_4] = - e_4,		&	[e_1, e_5] = - 5\, e_5,	&	[e_1, e_6] =  - e_6,	&	[e_1, e_7] = 3\, e_7,	\\ [2pt]
[e_2, e_3] = 3\, e_3,	&	[e_2, e_5] =  - 4\, e_5,	&	[e_2, e_6] = - e_6,		&	[e_2, e_7] =  2\, e_7,	&					\\ [2pt]	
[e_3, e_5] = -e_6,	&	[e_3, e_6] = - e_7.		&						&					&
\end{array}
\end{equation}

This Lie algebra is the semidirect product $\frh = \R^2\ltimes \frm$, where the abelian Lie algebra $\R^2=\langle e_1,e_2\rangle$ acts on the nilradical $\frm=\langle e_3,e_4,e_5,e_6,e_7\rangle$ 
of $\frh$ via the representation 
\[
\rho: \R^2 \rightarrow \Der(\frm),\quad \rho(e_1)=\diag(4,-1,-5,-1,3),\quad  \rho(e_2)=\diag(3,0,-4,-1,2).
\]
Moreover, $\frh$ is easily seen to be unimodular and completely solvable. 

We now show that the simply connected solvable Lie group $\H$ corresponding to $\frh$ does not contain any lattice. 
Indeed, the only nonzero terms in the descending central series of the nilradical $\frm$ of $\frh$ are 
\[
\frm^{\sst0}=\frm, \quad \frm^{\sst1} = \langle e_6,e_7\rangle,\quad \frm^{\sst2} =\langle e_7\rangle, 
\]  
and from \eqref{StrEqns} it follows that $\left. \ad_{e_1} \right |_{\frm^{\sst2}} = 3 \,  {\mathrm {Id}}_{\frm^{\sst2}}$. 
Hence, $\frh$ is not strongly unimodular.

Let $(e^1,\ldots,e^7)$ be the dual basis of $(e_1,\ldots, e_7)$. 
Then, the Chevalley-Eilenberg differentials of the covectors $e^i$ are the following:
\[
\renewcommand\arraystretch{1.4}
\begin{array}{lll}
de^1 = 0 = de^2, 			& de^3 = -4\,e^{13}-3\,e^{23},		& de^4 = e^{14}, \\
de^5 = 5\,e^{15}+4\,e^{25},	& de^6 = e^{16}+e^{26}+e^{35},	& de^7 = 	-3\,e^{17}-2\,e^{27}+e^{36}.
\end{array}
\]
Using these equations, we compute the cohomology groups of the Chevalley-Eilenberg complex $(\Lambda^\bullet(\frh^*),d)$ of $\frh$, obtaining:
\[
\renewcommand\arraystretch{1.4}
\begin{array}{lll} 
H^1 (\frg^*) = \langle  [e^1], [e^2] \rangle,	&	H^2 (\frg^*) = \langle  [e^1 \wedge e^2] \rangle,			&	H^3 (\frg^*) = H^4 (\frg^*) = \{ 0 \},\\[4pt]
H^5 (\frg^*) = \langle [e^{34567}] \rangle,	&	H^6 (\frg^*) = \langle [e^{234567}],  [e^{134567}] \rangle,	&	H^7 (\frg^*) = \langle [e^{1234567}] \rangle.
\end{array}
\]

Therefore, the Betti numbers of the Lie algebra $\frh$ are the following
\[
b_1(\frh)= b_6(\frh)  = 2,\qquad b_2(\frh)= b_5 (\frh) =  b_7 (\frh) = 1,\qquad b_3(\frh)=  b_4(\frh)=0.
\]

We now describe an explicit example of a left-invariant exact $\G_2$-structure $\f$ on $\H$.

Let us consider the following basis of $\frh^*$:
\[
\begin{aligned}
&E^1 = e^3,\quad E^2 = \frac{1}{4\sqrt{3}}\,e^2,\quad E^3 = -\sqrt{3}\,e^4 + 2\sqrt{3}\,e^6,& \\
&E^4 = e^5,\quad E^5 = e^4 + 2\,e^6,\quad E^6 = 8\sqrt{3}\,e^7, \quad E^7=e^1+\frac34\,e^2.& 
\end{aligned}
\]
The structure equations of $\frh$ with respect to this new basis are:
\[
\begin{dcases}
d E^1 = 4\, E^{17},\\
d E^2 =0,\\
d E^3 = -E^{37} + 2 \sqrt{3}\, E^{14} - \sqrt{3}\, E^{23} + 6\, E^{25},\\
d E^4 =  -5\, E^{47} + \sqrt{3} E^{24},\\
d E^5 =  -E^{57} + 2\, E^{14} + 2\, E^{23} - \sqrt{3}\, E^{25},\\
d E^6 =   3\, E^{67} + 2\, E^{13} + 2 \sqrt{3}\, E^{15} + \sqrt{3}\, E^{26},\\
d E^7 =0.
\end{dcases}
\]

We consider the $\G_2$-structure $\f$ on $\frh$ with adapted $\G_2$-basis $(E^1,\ldots,E^7)$, that is,
\[
\f = E^{127} + E^{347} + E^{567} + E^{135} - E^{146} - E^{236} - E^{245}.
\]
This is easily seen to be a closed, hence exact, $\G_2$-structure. In particular, we have
\[
\begin{aligned}
\f &= d \left ( - \frac 14\, E^{12} + \frac 16\, E^{34} - \frac 12\, E^{56} \right ). 
\end{aligned}
\]

\section*{Acknowledgements}
The first author was partially supported by MINECO-FEDER Grant PGC2018-098409-B-100,
and Gobierno Vasco Grant IT$1094-16$, Spain. 
The second and third author were supported by GNSAGA of INdAM and by  the project PRIN 2017  \lq \lq Real and Complex Manifolds: Topology, Geometry and Holomorphic Dynamics''.  
The authors are grateful to Thomas Bruun Madsen and Yves Cornulier for useful suggestions and conversations. 
The second and the third author would like to thank the organizers of the workshop  ``G$_2$ manifolds and related topics'' (Casa Matem\'atica Oaxaca, May 5-10, 2019), 
as well as the BIRS and CMO research center.

\appendix

\section{}\label{appendix}
In this appendix, we explicitly write the matrices associated to the endomorphisms $A_1, A_2, A_3,$ introduced in the proof of Theorem \ref{Thm23trivial}, in the case when $\frn=\frn_{\sst1}$ and $\frn=\frn_{\sst2}$. 
For each of these Lie algebras, we write an ordered basis of the cohomology group $H^k(\frn^*)$, $k=1,2,3,$ and the matrix associated to $A_k$ with respect to that basis. 
For the sake of convenience, we first impose the conditions on $D\in \Der(\frn)$ for which $\frg = \frn \rtimes_D\R$ is strongly unimodular. 

\subsection*{Case $\frn=\frn_{\sst1}$} 
For the Lie algebra $\frn_{\sst1}= \left(0,0,0,0,e^{12},e^{34}\right)$ with generic derivation $D_1$ given by \eqref{Dern1}, 
we have that $\frg=\frn\rtimes_{D_1}\R$ is strongly unimodular if and only if $a_8=-a_1-a_4-a_5$. In this case, we get
\begin{enumerate}[$\bullet$]
\item  $H^1(\frn_{\sst1}) = \langle[e^1],~[e^2],~[e^3],~[e^4]\rangle$, 
\[
A_1 = 
-\begin{pmatrix}
a_1	&	a_3	&	0	&	0		\\
a_2	&	a_4	&	0	&	0		\\
0	&	0	&	a_5	&	a_7		\\
0	&	0	&	a_6	&	-a_1-a_4-a_5	
\end{pmatrix};
\]
\item $H^2(\frn_{\sst1}) = \langle[e^{13}],[e^{14}],[e^{15}], [e^{23}],[e^{24}],[e^{25}],[e^{36}],[e^{46}] \rangle$,
\[A_2 = 
\begin{psmallmatrix}
-a_1 - a_5 &  -a_7 &  -a_{11} &  -a_3 & 0 & 0 &  a_{13} &  0  \\
 -a_6 &  a_4 +a_5. &  -a_{12} & 0 & -a_3 & 0 &  0 & a_{13} \\
 0 &  0 &  -2 a_1 -a_4  &  0 &  0 &  -a_3 & 0 &  0\\
  -a_2 &  0 & 0 &  -a_4 -a_5 &  -a_7 &  -a_{11} &  a_{14} &  0\\
0 & -a_2 &  0 &  -a_6  &a _1 +a_5 &  -a_{12} &  0 &  a_{14} \\
0 &  0 &  -a_2 &  0 &  0 &  -2 a_4 -a_1 &  0 &  0\\
0 &  0 & 0&  0 &  0&  0 & a_1+a_4 -a_5 &  -a_7 \\
 0 & 0 &  0 &  0& 0 &  0 &  -a_6 &  2 a_4 +2 a_1 +a_5
\end{psmallmatrix};
\]
\item $H^3(\frn_{\sst1}) = \langle [e^{125}], [e^{135}], [e^{136}], [e^{145}], [e^{146}], [e^{235}], [e^{236}], [e^{245}], [e^{246}], [e^{346}]   \rangle$, 
\[
A_3 =  
\begin{psmallmatrix}
2 a_4 +2 a_1 & 0& 0 & 0 & 0 & 0 &  0 &  0 & 0 & 0\\
0 &  2 a_1 +a_5 +a_4 &  0 &  a_6 &  0 & a_2 &  0 &  0 &  0 &  0\\
0 &  0 &  a_5 -a_4 &  0 & a_6 &  0 &  a_2  & 0 &  0 &  0\\
0 &  a_7 & 0 & a_1 -a_5 &  0 &  0 &  0 &  a_2 & 0 &  0\\
0 &  0 &  a_7 & 0 &  -2 a_4 -a_5 -a_1 &  0 &  0 & 0 &  a_2 &  0\\
0 & a_3 &  0 &  0 &  0 &  2 a_4 +a_5 +a_1 &  0 &  a_6 &  0 &  0\\
0 & 0 & a_3 &  0 & 0 & 0 & a_5 -a_1 &  0 & a_6 &  0\\
0 &  0&  0 &  a_3 &  0 &  a_7 & 0 &  a_4 -a_5 &  0& 0\\
0 &  0 & 0 &  0 &  a_3 &  0 & a_7 &  0 & -2 a_1 -a_5 -a_4 & 0\\
0 & 0 & 0 & 0 &  0 &  0 &  0 &  0 &  0 &  -2 a_4 -2 a_1 
\end{psmallmatrix}.
\]
\end{enumerate}

\subsection*{Case $\frn=\frn_{\sst2}$}  
Consider the Lie algebra $\frn_{\sst2}=\left(0,0,0,0,e^{13}-e^{24},e^{14}+e^{23}\right)$ with generic derivation $D_2$ given by \eqref{Dern2}. 
Then, $\frg=\frn_{\sst2}\rtimes_{D_2}\R$ is strongly unimodular if and only if $a_7=-a_1$. In this case, we have
\begin{enumerate}[$\bullet$]
\item $H^1(\frn_{\sst2}) = \langle  [e^1],~[e^2],~[e^3],~[e^4] \rangle $,
\[
A_1 = 
\begin{pmatrix}  
-a_1 &  a_2 &  -a_5 & a_6\\
-a_2 &  -a_1 & -a_6 &  -a_5\\
-a_3 &  a_4 &  a_1 &  a_8\\
-a_4 &  -a_3 & -a_8 &  a_1
\end{pmatrix};
\]
\item $H^2(\frn_{\sst2}) = \langle [e^{12}], [e^{34}], [e^{13}+e^{24}],  [e^{14}-e^{23}],  [e^{16}+e^{25}],  [-e^{15}+e^{26}],  [e^{36}+e^{45}],  [-e^{35}+e^{46}]  \rangle $,
\[
A_2 =\begin{psmallmatrix}
-2 a_1&  -2 a_6 &  -2 a_5 & a_9-a_{14} & a_{10}+a_{13} &  0 &  0 &  0\\
a_4 &  0 &  a_8-a_2 & - \frac 12  a_{15} & \frac 12 a_{11}- \frac 12 a_{16} & -a_6 &  \frac 12 a_{13}+ \frac 12 a_{10} &  - \frac 12 a_{9}+ \frac 12 a_{14}\\
 -a_3 & a_2-a_8 &  0 &  - \frac 12 a_{16} &   \frac 12 a_{12}+ \frac 12 a_{15} &  -a[
_5 &  - \frac 12 a_{14}+ \frac 12 a_9 &  \frac 12 a_{13}+ \frac 12 a_{10}\\
0 &   &  0 &  -a_1 &  a_8+2 a_2 &  0 &  -a_5 &  a_6\\
0 &  0 &  0 &  -a_8-2 a_2  &  -a_1 & 0 &  -a_6 &  -a_5\\
0 &  2 a_4  &  -2 a_3  &  0 &  0 &  2 a_1  & a_{11} -a_{16} & a_{12} +a_{15}\\
0 &  0 &  0 &  -a_3 &  a_4  & 0 &  a_1 &  a_2 +2 a_8\\
0 & 0 &  0 &  -a_4 &  -a_3 &  0 &  -a_2 -2 a_8 &  a_1
 \end{psmallmatrix}; 
\]
\item $H^3(\frn_{\sst2}) = \langle [e^{125}],  [e^{126}],  [e^{345}],  [e^{346}],  [-e^{135}+e^{146}], [-e^{135}+e^{236}], [e^{135}+e^{245}], [e^{136}+e^{246}], [e^{145}-e^{246}], [e^{235}-e^{246}]   \rangle $, 
\[
A_3 =  \begin{psmallmatrix}
-2 a_1 & 0& 0 &  -a_4 &  -a_4&  0 &  -a_3 & 0 &  -a_4 &  a_2+a_8\\
0 &  -2a_1 &  -a_2-a_8  &  a_5 & -a_6 &  -2 a_5 &  -a_6 &  2 a_6 &  -a_5 &  0\\
0 &  a_2+a_8 &  2 a_1& a_6 & a_5&  0 &  -a_5 &  0 &  -a_6 &  0\\
a_6 &  0 & a_4& 0 &  -a_2 -2 a_8 &  0 &  -a_8 -2 a_2 &  a_2+a_8& 0 &  0\\
a_5&  0 &  a_3 &  a_2+2a_8 &  0 &  -a_2-a_8&  0 &  0 &  -a_8-2 a_2 &  0\\
0 & -a_3 &  0 &  0 &  -a_2&  0 &  -a_8 -2 a_2 &  a_2-a_8 & 0 &  -a_5\\
-a_5 & 0 &  -a_3&  a_2 &  0 &  a_2+a_8 &  0 &  0 &  -a_8 &  0\\
0 & -a_4 &  0 &  a_2 &  0 &  -a_2+a_8 &  0 &  0 & -a_8 -2 a_2 & a_6\\
-a_6 & 0 & a_4 &  0 &  a_2 &  0 & a_8 &  a_2+a_8&  0 &  0\\
-a_2-a_8 &  0 &  0 & a_3 &  a_3 &  -2 a_3 &  a_4&  -2 a_4 &  -a_3 &  -2 a_1
\end{psmallmatrix}.
\]
\end{enumerate}

\end{document}